\documentclass[11pt,reqno]{amsart}

\usepackage{hyperref}
\usepackage{morefloats}

\usepackage{assorted}

\newcommand{\bluefootnote}[1]{
\relax
}
\DeclareRobustCommand{\note}[1]{%
\relax
}

\newcommand{\defnfont}[1]{\textbf{#1}}

\usepackage[backend=bibtex,hyperref=true,maxbibnames=5,style=numeric,firstinits=true,sorting=debug]{biblatex}
\makeatletter
\renewcommand{\theenumi}{\roman{enumi}}

\renewcommand{\p@enumii}{\theenumi.}
\makeatother

\title[Rank-Finiteness for Modular Categories]{Rank-Finiteness for Modular Categories}
\author{Paul Bruillard}
\email{pjb2357@gmail.com}
\address{Department of Mathematics,
    Texas A\&M University,
    College Station, TX
    U.S.A.}
\curraddr{Pacific Northwest National Laboratory, 902 Battelle Boulevard,
Richland, WA U.S.A}

\author{Siu-Hung Ng}
\email{rng@math.lsu.edu}

\address{Department of Mathematics, Louisiana State University, Baton Rouge, LA
    U.S.A.}
\thanks{The first, third and fourth authors were partially supported by
NSF grant DMS1108725, and the second author was partially supported by NSF grants
DMS1001566, DMS1303253, and DMS1501179. This project began at a workshop at the American Institute of Mathematics, and was completed while E.R. was visiting the Beijing International Center for Mathematical Research, Peking University--the support and hospitality of both institutions are gratefully acknowledged. It is a pleasure to thank Vladimir Turaev and Pavel Etingof for valuable comments.}

\author{Eric C. Rowell}
\email{rowell@math.tamu.edu}
\address{Department of Mathematics,
    Texas A\&M University,
    College Station, TX
    U.S.A.}

\author{Zhenghan Wang}
\email{zhenghwa@microsoft.com}
\address{Microsoft Research Station Q and Department of Mathematics,
    University of California,
    Santa Barbara, CA
    U.S.A.}


\date{\today}
\dedicatory{}

\begin{document}
\raggedbottom
\begin{abstract}
  We prove a rank-finiteness conjecture for modular categories: up to equivalence, there are
   only finitely many modular categories of any fixed rank.  Our technical advance is a generalization of the Cauchy theorem in
  group theory to the context of spherical fusion categories.  For a modular category $\CC$ with $N=\ord(T)$, the order of the modular $T$-matrix, the Cauchy theorem says that the set of primes dividing the global quantum
  dimension $D^2$ in the Dedekind domain $\mathbb{Z}[e^{\frac{2\pi i}{N}}]$ is identical
  to that of $N$.
\end{abstract}

\maketitle

\section{Introduction}
  \label{Introduction}
Modular categories are intricate organizing algebraic structures appearing in a
  variety of mathematical subjects including topological quantum field theory
  \cite{Tu1}, conformal field theory \cite{MS1}, representation theory of
  quantum groups \cite{BKi}, von Neumann algebras \cite{EK1}, and vertex operator
  algebras \cite{Hu1}.  They are fusion categories with additional braiding and
  pivotal structures \cite{ENO1,Tu1,BKi} satisfying a non-degeneracy condition.  These extra structures endow them with some
  \lq\lq abelian-ness" which makes the theory of modular categories easier.

Besides the intrinsic mathematical aesthetics, another motivation for studying
 modular categories comes from their application in
  condensed matter physics and quantum computing \cite{W1, W2}.  Unitary
  modular categories are algebraic models of anyons in two dimensional
  topological phases of matter (where simple objects model anyons).  In topological
  quantum computation,  anyons give rise to quantum computational models.
  Modular categories  have also been used recently to construct physically
  realistic three dimensional topological insulators and superconductors
  \cite{WW1, BCFV1}.  Therefore, modular categories form part of the mathematical foundations of topological quantum computation.

A modular category $\mathcal{C}$ over $\mathbb{C}$ is a non-degenerate braided spherical fusion category over $\mathbb{C}$
  \cite{Tu1, BKi}.  A \defnfont{fusion category} $\CC$ over $\mathbb{C}$ is an abelian $\BC$-linear semisimple rigid
  monoidal category with a simple unit object $\1$, finite-dimensional morphism spaces and finitely many isomorphism classes of simple objects.  The label set $\Pi_\CC$ of $\CC$ is the set of isomorphism classes of simple objects of the modular category $\CC$. The \defnfont{rank} of $\CC$ is the finite number $r=|\Pi_\CC|$.  Each modular category $\mathcal{C}$ leads to a $(2+1)$-dimensional
  topological  quantum field theory and, in particular, colored
  oriented framed link invariants \cite{Tu1}.  The invariant $d_i$ for the circle
  colored by a label $i\in \Pi_\CC$ is called the \defnfont{quantum
  dimension} of the label $i$, and the global quantum dimension $D^2=\sum_{i\in \Pi_\CC}d_i^2$  is an important invariant of
  $\mathcal{C}$.  The invariant of a Hopf
  link colored by labels $i,j$ will be denoted as $S_{ij}$, which forms the unnormalized modular $S$-matrix $S=(S_{ij}), i,j\in \Pi_\CC$. The invariant of the circle with a right-handed
  kink colored by a label $i$ is $\theta_i\cdot d_i$ for some root of unity $\theta_i$,
  which is called the \defnfont{topological twist} of the label $i$.  The topological
  twists are encoded by a diagonal modular $T$-matrix $T=(\delta_{ij}\theta_i), i,j\in \Pi_\CC$.   The
  modular $S$-matrix and modular $T$-matrix together lead to a projective representation of the
  modular group $\SL{}$ by sending the generating matrices

  \begin{equation*}
     \fs=\mtx{0 & -1\\ 1 & 0},\qquad \ft=\mtx{1 & 1\\ 0 & 1}
  \end{equation*}
  to $S,T$, respectively \cite{Tu1, BKi}.  The
  modular group $\SL{}$ arises here as the mapping class group of the torus.  Amazingly, the kernel of this projective representation of $\SL{}$ is always a congruence
  subgroup of $\SL{}$ \cite{NS3}.  The $S$-matrix determines the fusion rules through the Verlinde formula, and the $T$-matrix is of finite order $\ord(T)$ by
  Vafa's theorem \cite{BKi}.  Together, the pair $(S,T)$ is called the \emph{modular data} of $\CC$.

The abelian-ness of modular categories first
  manifests itself in the tensor product, via the braiding.  But a deeper sense of
  abelian-ness is hidden in the
  Galois group of the number field $\mbbK_{\mathcal{C}}=\mathbb{Q}[S_{ij}],i,j\in \Pi_\CC$:  $\Gal(\mbbK_{\mathcal{C}}/\BQ)$ is isomorphic to an
    abelian subgroup \cite{BG1, RSW} of the symmetric group $\mathfrak{S}_r$, where $r$ is the rank of
  $\mathcal{C}$.  This profound abelian-ness permits the application of number theory to
  the study of modular categories.

In $2003$, the fourth author conjectured that, up to equivalence, there are only
  finitely many modular categories of a given rank, which we will call {\it the
  rank-finiteness conjecture} \cite{W3,RSW}.

The first main result of the paper is a proof of this conjecture:
  \begin{reptheorem}{Rank Finiteness}
    There are only finitely many modular categories of fixed rank $r$, up to
    equivalence.
  \end{reptheorem}

The idea of the proof is as follows.  A classical result of Landau \cite{L1} is that for each $n$, there are only finitely many finite groups $G$ with exactly $n$ irreducible complex representations.  Landau's result is proved by dividing the class equation $|G|=\sum_{i=1}^n [G:C(g_i)]$ by $|G|$ to produce the Diophantine equation $1=\sum_{i=1}^n\frac{1}{x_i}$.  Observing that this equation has finitely many solutions in the positive integers, one obtains a bound on $|G|$, which implies the result.  Our proof follows the same strategy, with the dimension equation $D^2=\sum_{i\in \Pi_\CC}d_i^2$ playing the role of the class equation.
Firstly, Ocneanu rigidity implies that for a fixed set of fusion rules, there are
  only finitely many equivalence classes of modular categories \cite{ENO1}.
  Hence, the rank-finiteness conjecture is reduced to showing that there are only
  finitely many possible fusion rules for any given rank. Using the Verlinde
  formula, we can deduce the finiteness of
  fusion rules for a given rank from a bound of the global quantum dimension
  $D^2$. In particular, if there were only finitely possible values
  of $D^2$ in each rank, then rank-finiteness would follow \cite[Prop. 6.2]{RSW}.  Since the $\{d_i\}$ are algebraic integers in a cyclotomic field, not necessarily in $\mathbb{Q}$, a new approach is needed.

Recall that the Frobenius-Schur exponent $\FSexp(\CC)$ of a spherical fusion category $\CC$ over $\BC$ is equal to the order of the $T$-matrix of its (modular) Drinfeld center $Z(\CC)$, and for a modular category, $\FSexp(\CC)=\ord(T)$, the order of the $T$-matrix (cf. \cite{NS2}).
 Our technical advance is the second main result, which is of independent interest.  This is a generalization of the Cauchy theorem in
  group theory to the context of spherical fusion categories:

\begin{reptheorem}{Cauchy Theorem for Modular Categories}
    If $\CC$ is a spherical fusion category over $\BC$ with $N=\FSexp(\CC)$, then the set of prime ideals dividing the principal ideal generated by the global quantum dimension $D^2$ of $\CC$ in the Dedekind domain  $\BZ[e^{2\pi i/N}]$ is identical to that of $N$.
\end{reptheorem}

As a consequence, the quantum dimensions $\{d_i\}$\and $D^2$ of a modular category have special arithmetic properties: they are
  so-called {\it $\mcS$-units} with respect to the common set $\mcS$ of prime ideals in the
  factorization of the ideals generated by $D^2$ and $\ord(T)$.   Then regarding
  $D^2=\sum_{i\in\Pi_\CC}d_i^2$ as an $\mcS$-unit equation, we apply a powerful theorem of Evertse \cite{Ev1}: for any fixed $r$, there are finitely many non-degenerate solutions to $1=\sum_{i=1}^r x_i$ for $\mcS$-units $x_i$.  It follows that there are
  only finitely many solutions to the dimension equation $D^{2}=\sum_{i\in\Pi_\CC}d_{i}^{2}$, which, in turn, bounds $D^2$ in terms of the rank.  Rank-finiteness then follows.

All these steps can be made effective, so we
  have explicit bounds for the number of solutions to the dimension equation.  The bound for the number of
  possible modular
  categories for a given rank that we obtained is absurdly large.  For example
  for rank=$2$, there are only $8$ modular categories up to equivalence while, our bound for the number of solutions to the $\mcS$-unit
  equation $D^{2}=1+d_1^2$ (see Proposition \ref{201302051706}) is
   $2^{8.15885\times 10^{41}}$.
  A natural question is to determine if there is a better bound for the
  number of modular categories of rank=$r$.  Etingof observes in Remark \ref{EtRem} that the number of modular categories of rank=$r$ grows faster than any polynomial in $r$.

The rank-finiteness conjecture was motivated by the classification of topological phases of matter \cite{W1,W2}.  Topological phases of matter
  are states of matter which have an energy gap in the thermodynamic limit and
  are stable under small yet arbitrary perturbations.  Thus, they cannot be
  continuously deformed non-trivially inside topological states of matter.  A mathematical
  manifestation of this rigidity of topological phases of matter should be rank-finiteness.
The rank-finiteness theorem implies that, in principle, modular categories can be classified for low rank cases.  Indeed, early progress in the classification program is the complete classification of unitary modular categories of rank at most $4$ \cite{RSW}. The authors of this paper have now pushed the classification to rank $5$, which will appear in a separate publication.

The Cauchy theorem for spherical fusion categories is a generalization of \cite[Thm. 8.4]{NS2} for \emph{integral} fusion categories. This question was originally asked by Etingof and Gelaki in the context of Hopf algebras \cite[Question 5.1]{EG1}, and subsequently verified for Hopf algebras in \cite{KSZ1}. Moreover, the Cauchy theorem provides an affirmative answer to \cite[Question 6.10]{DLN1}.

The proof of the Cauchy theorem for spherical fusion categories relies heavily upon \emph{higher Frobenius-Schur ($FS$)-indicators}, which are discussed in Section \ref{Modular Categories}.  The theory of $FS$-indicators has played a key role in several recent results such as the congruence subgroup theorem \cite{NS2} and Galois symmetry \cite{DLN1} (cf. Sec. \ref{galsymsection}).

For a finite group $G$, the $n$th $FS$-indicator of a representation $V$ over $\BC$ with character $\chi_V$ is given by $\nu_n(V):=\frac{1}{|G|}\sum_{g\in G}\chi_V(g^n)$. The Frobenius-Schur theorem asserts that the second $FS$-indicator $\nu_2(V)$ of an irreducible representation $V$ must be $-1$, $0$ or $1$.

The second $FS$-indicator for a primary field of a rational conformal field theory is introduced by Bantay \cite{Ban1} as an expression in terms of the associated modular data. Bantay's expression provides a formula for the second $FS$-indicator of a simple object in a modular category.  Second $FS$-indicators are later introduced by Fuchs, Ganchev,
Szlach\'anyi, and Vescerny\'es for certain ${\mathbb{C}}^*$-fusion categories \cite{FGSV}, by Linchenko and Montgomery \cite{LM} for semisimple Hopf algebras, and by Mason and Ng \cite{MNg} for semisimple quasi-Hopf algebras.  In each case, they are shown to satisfy an analogue of the Frobenius-Schur theorem.  The less familiar higher $FS$-indicators ($n>2$) are also defined for semisimple Hopf algebras in \cite{LM}, and studied extensively by Kashina, Sommerh\"auser and Zhu in \cite{KSZ1}.

The second $FS$-indicators are first shown in \cite{MNg} to be invariants of the integral fusion categories $\Rep(H)$, where $H$ is a semisimple quasi-Hopf algebra over $\BC$. This result motivates the development of higher $FS$-indicators $\nu_n$ for pivotal categories by Ng and Schauenburg \cite{NS1}, and in particular, spherical fusion categories $\CC$ over $\BC$ \cite{NS2}. The invariance of higher $FS$-indicators for $\Rep(H)$  \cite{NS4} follows from their categorical treatment.
For a spherical fusion category $\CC$, there exists a minimal positive integer $N=\FSexp(\CC)$, called the \emph{$FS$-exponent} of $\CC$, which satisfies $\nu_N(X_k)=d_k$ for all $k\in \Pi_\CC, X_k\in k$ \cite{NS2}. The $FS$-exponent of a spherical fusion category behaves, in many ways, like the exponent of a finite group. In fact, the $FS$-exponent of $\Rep(G)$ for any finite group $G$ is equal to the exponent of $G$.

Inspired by the paper \cite{SZ1} of Sommerh\"auser and Zhu, Ng and Schauenburg formulate a categorical definition of generalized $FS$-indicators in \cite{NS3} so that $\SL{}$ acts on them in two different but compatible ways. These generalized $FS$-indicators reveal new arithmetic properties of modular categories, which include the congruence kernels of their projective representations of $\SL{}$ (\emph{loc.~cit.}) and the Galois symmetry \cite{DLN1} conjectured by Coste and Gannon \cite{CG1}.

 Our reduction of rank-finiteness to Evertse's theorem obscures the nature of rank-finiteness for modular categories.  The
  key to Evertse's finiteness of $\mcS$-unit solutions is the Schmidt subspace
  theorem, which implies finiteness theorems for some simultaneous approximations
  to algebraic numbers by elements of a number field. A more direct proof of rank-finiteness might shed light on whether or not rank-finiteness also holds for spherical fusion categories.  One potential approach is taking the Drinfeld center of spherical fusion categories and then deducing rank-finiteness for spherical fusion categories from the modular case.  The key hurdle to this approach is controlling the rank of the Drinfeld center in terms of the rank of the original category.  In light of \cite{Eti2}, this appears to be difficult: even for integral fusion categories $\CC$, the rank of the Drinfeld center $Z(\CC)$ is super-polynomial in the rank of $\CC$.

The contents of the paper are as follows.  Section \ref{Modular Categories} is a collection of necessary results on fusion and modular categories.  In Section \ref{Rank Finiteness and the Cauchy Theorem}, we prove the Cauchy theorem for spherical fusion categories and rank-finiteness for modular categories, as well as for modularizable premodular categories.  In Section \ref{asymptotics}, we conclude with a discussion of asymptotics and future directions.

\section{Modular Categories}
  \label{Modular Categories}
In this section, we will collect some conventions and essential results on
spherical fusion categories and modular categories. Most of these results can be
found in  \cite{Tu1, BKi, ENO1, NS1, NS2, NS3} and the references
therein.  All fusion and modular categories are over the complex numbers $\mathbb{C}$ in this paper unless stated otherwise.

\subsection{Basic Definitions}
\label{subsection: Basic Definitions}

 A modular
category is a braided spherical fusion category in which the braiding is
non-degenerate.
Modular categories were first axiomatized by Turaev \cite{Tu2}, based
on earlier notions in rational conformal field theory by Moore and Seiberg \cite{MS1} and related foundational work of Joyal and Street \cite{JS1}.  Early interesting examples arose in the work of Reshetikhin and Turaev \cite{ReshTur} on
quantum groups and their application to low-dimensional topology.   In this section, we will give the precise definition and
describe some further properties and consequences of the definition.


  \subsubsection{Fusion Categories}
  \label{subsubsection: Fusion Categories}

    Recall from \cite{ENO1}, a \defnfont{fusion category}  $\CC$ (over $\BC$) is an abelian
    $\BC$-linear semisimple rigid
    monoidal category with a simple unit object $\1$, finite-dimensional morphism
    spaces and finitely many isomorphism classes of simple objects.

    Here a \defnfont{monoidal} category is a category $\CC$ with
    \begin{enumerate}
     \item  a bifuntor $\ot:\CC\times\CC\rightarrow \CC$
     \item a natural  isomorphism (associativity) $\alpha_{UVW}:(U\ot V)\ot W\rightarrow U\ot(V\ot W)$
     \item and a unit object $\1$ with natural isomorphisms $\lambda_V:\1\ot V\rightarrow V$ and $\rho_V:V\ot\1\rightarrow V$
    \end{enumerate} so that the associativity and the unit satisfy the pentagon and triangle compatibility axioms \cite{Mac}.  A monoidal functor is a pair $(F, \phi)$ where $F:\CC\rightarrow\DD$ is a functor with $F(\1_\CC)\cong \1_\DD$ and $\phi_{V,W}: F(V \otimes W) \rightarrow F(V) \otimes F(W)$ is a natural isomorphism which is compatible with the associator.  Two monoidal categories $\CC$ and $\DD$ are equivalent if there is a monoidal functor $(F, \phi)$ from $\CC$ to $\DD$ so that the functor $F: \CC \rightarrow \DD$ is an equivalence of the underlying ordinary categories \cite{Mac}.

    In a fusion category $\CC$ with tensor product $\o$ and unit
    object $\1$, the left dual of $V \in \CC$ is a triple $(V^{*}, \db_V, \ev_V)$,
    where $\db_V: \1 \to V \o V^{*}$ and $\ev_V: V^{*} \o V \to \1$  are the
     coevaluation and evaluation morphisms.  The left duality can be
    extended to a monoidal functor $(-)^{*} : \CC \to
    \CC^{\op}$, and so $(-)\bidu: \CC \to \CC$ defines a monoidal
    equivalence.
    Moreover, we can choose $\1^{*} = \1$. The linear space of morphisms
    between objects $V$ and $W$ will be denoted as
    $\Hom_{\mcC}\(V,W\)$.  Right duals are similarly defined,

    Let $\Pi_\CC$ be the set of isomorphism classes of simple objects of the
    fusion category $\CC$. The \defnfont{rank} of $\CC$ is the finite number
    $r=|\Pi_\CC|$, and we denote the members of $\Pi_\CC$ by $\{0,\ldots,r-1\}$.
    We simply write $V_i$ for an object in the isomorphism class $i \in
    \Pi_\CC$. By convention, the isomorphism class of $\1$ corresponds to $0\in\Pi_\CC$. The
    rigidity of $\CC$ defines an involutive permutation $i \mapsto i^*$ on $\Pi_\CC$,
    which is given by $V_{i^*} \cong V_i^{*}$
    for all $i \in \Pi_\CC$.

  \subsubsection{Braidings}
  \label{subsubsection: Braiding and Drinfeld Isomorphisms}

    A \defnfont{braiding} $c$ of a
    fusion category $\CC$ is a natural family of isomorphisms $c_{V, W}: V \o W \to
    W \o V$ in $V$ and $W$ of $\CC$ which satisfy the hexagon axioms:
    \begin{equation*}
      \xymatrix{ U \o (V \o W) \ar[r]^-{c_{U, V\o W}} &   (V \o W) \o U
      \ar[d]^-{\a} \\
      (U \o V) \o W \ar[d]_-{c \o \id}\ar[u]^-{\a} &  V \o (W \o U)  \\
      (V \o U) \o W \ar[r]_-{\a} &  V \o (U \o W) \ar[u]_-{\id \o c}
      \,,}\quad
      \xymatrix{ (U \o V) \o W \ar[r]^-{c_{U\o V, W}} &   W \o (U \o V)
      \ar[d]^-{\a\inv} \\
      U \o (V \o W) \ar[d]_-{\id \o c}\ar[u]^-{\a\inv} &  (W \o U) \o V  \\
      U \o (W \o V) \ar[r]_-{\a\inv} &   (U \o W) \o V \ar[u]_-{c \o \id}\,,
      }
    \end{equation*}

    for all $U, V, W \in \CC$, where $\a$ is the associativity isomorphism of $\CC$
    (cf. \cite{JS1}).

    A \defnfont{braided fusion category} is a pair $(\CC, c)$ in which $c$ is a braiding
    of the fusion category $\CC$. We simply call $\CC$ a braided fusion
    category if the underlying braiding $c$ is understood.

  \subsubsection{Spherical Fusion Categories}
  \label{subsubsection: Spherical Fusion Categories}

    A \defnfont{pivotal structure} of a fusion category $\CC$ is an isomorphism
    $j:\Id_\CC \to (-)\bidu$ of monoidal functors. One can respectively define the  left and the right
    \defnfont{pivotal traces} of an endomorphism $f: V \to V$ in $\CC$ as
    \begin{align*}
      \ptrl(f)&=\(\1\xrightarrow{\db_{V^{*}}} V^{*}\o V\bidu\xrightarrow{\id\o j_V\inv} V^{*}\o
      V\xrightarrow{\id \o f} V^{*}\o V \xrightarrow{\ev_V}\1\)\\
      \ptrr(f)&=\(\1\xrightarrow{\db_V}V\o
      V^{*}\xrightarrow{f\o \id}V\o V^{*} \xrightarrow{j_V\o \id} V\bidu \o V^{*}
      \xrightarrow{\ev_{V^{*}}}\1\)\,.
    \end{align*}

    Note that $j_V\du = j_{V\du}\inv$ (cf. \cite[Prop. A.1]{Sch}), and so we have $\ptrl(f) = \ptrr(f\du)$.
    Since $\1$ is a simple object of $\CC$, both pivotal traces $\ptrl(f)$ and
    $\ptrr(f)$ can be identified with some scalars in $\BC$.  A pivotal structure
    on $\CC$ is called \defnfont{spherical} if the
    two pivotal traces coincide for all endomorphisms $f$ in $\CC$. In a spherical category, the pivotal trace(s)
    will be denoted by $\ptr(f)$.

    For the purpose of this paper, a \defnfont{pivotal} (resp. \defnfont{spherical})
    \defnfont{category} $(\CC, j)$ is a fusion category $\CC$ equipped with a pivotal
    (resp. spherical) structure $j$. We will denote the pair $(\CC, j)$ by
    $\CC$ when there is no ambiguity. The \defnfont{left} and the \defnfont{right
    pivotal dimensions} of $V\in\CC$ are defined as $d^\ell(V)=\ptr^\ell(\id_V)$ and
    $d^{r}(V)=\ptr^r(\id_V)$ respectively.

  \subsubsection{Modular Categories}
  \label{subsubsection: Modular Categories}
    Following \cite{Ka1}, a \defnfont{twist}  (or \defnfont{ribbon structure}) of a
    braided fusion category $(\CC, c)$ is an $\BC$-linear automorphism,
    $\theta$, of $\Id_\CC$ which satisfies
    \begin{equation*}
      \theta_{V \o W} = (\theta_V \o \theta_W)\circ c_{W, V}\circ  c_{V, W}, \quad \theta^{*}_V = \theta_{V^{*}}
    \end{equation*}

    for $V, W \in \CC$. A braided fusion category equipped with a ribbon structure
    is called a \defnfont{ribbon fusion} or \defnfont{premodular} category.

     Associated with the braiding $c$ is an isomorphism of $\BC$-linear functors
    $u: \Id_\CC \to (-)\bidu$, called the \defnfont{Drinfeld isomorphism}.  When $\CC$ is a strict
    fusion category, $u_V$ is the composition:
    \begin{equation*}
      u_V :=(V \xrightarrow{\db \o \id} V^{*} \o V\bidu \o V \xrightarrow{\id \o
      c\inv} V^{*} \o V \o V\bidu   \xrightarrow{\ev \o \id } V\bidu=V )\,.
    \end{equation*}

    If $u$ is the Drinfeld isomorphism associated with $c$, and $\th$ is a ribbon
    structure, then
    \begin{equation}
      j=u \theta
    \end{equation}

    is a spherical  structure of $\CC$. This equality defines a one-to-one
    correspondence between the spherical structures and the ribbon structures on
    $(\CC,c)$. In particular, every premodular category admits a spherical structure.

    A premodular category $\CC$ is called a \defnfont{modular category} if the
    $S$-\defnfont{matrix} of $\CC$, defined by
    \begin{equation*}
      S_{ij} = \ptr(c_{V_j, V_{i^*}} \circ c_{V_{i^*}, V_j}) \text{ for }  i, j \in \Pi_\CC\,,
    \end{equation*}

    is non-singular.  Note that $S$ is a symmetric matrix and that
    $d^{r}\(V_i\)=S_{0i}=S_{i0}$ for all $i$.

\subsection{Further Properties and Basic Invariants}
\label{subsection: Further Consequences and Properties}
  \subsubsection{Grothendieck Ring and Dimensions}
    The \defnfont{Grothendieck ring} $K_0(\CC)$ of a fusion category $\CC$ is the
    $\BZ$-ring generated by $\Pi_\CC$ with multiplication induced from
    $\o$.  The structure coefficients of $K_0(\CC)$ are obtained from:
    \begin{equation*}
      V_i \o V_j \cong \bigoplus_{k \in \Pi_\CC} N_{i,j}^k \,V_k
    \end{equation*}

    where $N_{i,j}^k = \dim(\Hom_{\mcC}\(V_k, V_i\o V_j\))$. This family of non-negative
    integers $\{N_{i,j}^k\}_{i,j,k \in \Pi_\CC}$ is called the \textit{fusion
    rules} of $\CC$.

    In a braided fusion category, $K_0(\CC)$ is a commutative ring and the fusion
    rules satisfy the symmetries:
    \begin{equation}
      \label{fusion symmetries}
      N_{i,j}^k=N_{j,i}^k=N_{i,k^*}^{j^*}=N_{i^*,j^*}^{k^*},\quad
      N_{i,j}^0=\delta_{i,j^*}.
    \end{equation}

    The \defnfont{fusion matrix} $N_i$ associated to $V_i$, defined by
    $(N_i)_{k,j}=N_{i,j}^k$, is
    an integral matrix with non-negative entries. In the braided fusion setting,
    these matrices are normal and mutually commuting. The largest real eigenvalue of
    $N_i$ is called the \defnfont{Frobenius-Perron
    dimension} of $V_i$ and is denoted by $\FPdim(V_i)$.
    Moreover, $\FPdim$ can be extended to a $\BZ$-ring homomorphism
    from $K_0(\CC)$ to $\BR$ and is the unique such homomorphism that is positive
    (real-valued) on $\Pi_\CC$ (see \cite{ENO1}). The \defnfont{Frobenius-Perron dimension} of
    $\CC$ is defined as
    \begin{equation*}
      \FPdim(\CC) = \sum_{i \in \Pi_\CC} \FPdim(V_i)^2\,.
    \end{equation*}

    \begin{defn}
      A fusion category $\mcC$ is said to be
      \begin{enumerate}
      \item \defnfont{weakly integral} if $\FPdim\(\mcC\)\in\mbbZ$.
        \item \defnfont{integral} if $\FPdim\(V_{j}\)\in\mbbZ$ for all
        $j\in\P_{\mcC}$.
        \item \defnfont{pointed} if $\FPdim\(V_{j}\)=1$ for all
        $j\in\P_{\mcC}$.
      \end{enumerate}

      Furthermore, if $\FPdim\(V\)=1$, then $V$ is \defnfont{invertible}.
    \end{defn}

    \begin{rmk}
      The terminology \textit{invertible} arises from the fact that $\FPdim\(V\)=1$
      if and only if $V\otimes V^{*}\cong\1$.  The set of invertible simple
objects generates a full fusion subcategory $\CC_{pt}$ called the \defnfont{pointed
subcategory}.
    \end{rmk}

    Let $\CC$ be a pivotal category. It follows from \cite[Prop. 2.9]{ENO1} that
    $d^{r}(V\du)=\ol{d^{r}(V)}$ is an algebraic integer for any $V \in \CC$. The
    \defnfont{global dimension} of $\CC$ is defined by
    \begin{equation*}
      D^2 = \sum_{i \in \Pi_\CC} |d^{r}(V_i)|^2.
    \end{equation*}

    \begin{remark}
      It is worth noting that the global dimension $D^2$ can be defined for any
      fusion category (cf. \cite{ENO1}), and does not depend on the existence,
      or choice of, of a pivotal structure.
    \end{remark}

    By \cite{M2, ENO1}, a pivotal structure of a fusion category $\CC$ is spherical
    if, and only if, $d^{r}(V)$ is real for all $V \in \CC$. In this case,
    $d^{r}\(V\)=d^{\ell}\(V\)$ and we simply write $d\(V\)$ to refer to the
    dimension of $V$. Furthermore for $i\in\P_{\mcC}$, we adopt the shorthand
    $d_{i}=d\(V_{i}\)$.

    A fusion category $\CC$ is called \defnfont{pseudo-unitary} if $D^{2}=\FPdim(\CC)$.
    For a pseudo-unitary fusion category $\CC$, it has been
    shown in \cite{ENO1} that there exists a unique spherical structure of $\CC$
    such that $d\(V\) = \FPdim(V)$ for all objects $V \in \CC$.

  \subsubsection{Spherical and Ribbon Structures}

    The set of isomorphism classes of invertible objects $G(\CC)$ in a fusion category $\CC$ forms a group in
    $K_0(\CC)$ where $i\inv = i^*$ for $i \in G(\CC)$. For modular categories $\CC$, the group
    $G(\CC)$ parameterizes pivotal structures on the underlying braided
    fusion category:\footnote{The second part of this result
    was pointed out to us by Naidu.}
    \begin{lem}
      \label{pivotalinvertible}
      Let $\CC$ be a modular category. There is a bijective correspondence between
      the pivotal structures of the underlying braided fusion category $\CC$ and
      the group of invertible objects $G(\CC)$. Under this correspondence, the
      inequivalent spherical structures of $\CC$ map onto the maximal elementary
      abelian 2-subgroup, $\Omega_2 G(\CC)$, of $G(\CC)$.
    \end{lem}
    \begin{proof}
      Let $j_0$ be the spherical structure of the modular category $\CC$.
      For any pivotal structure $j$ of $\CC$, we have $j_0\inv j \in  \Aut_\o(\Id_\CC)$,
      the group of automorphisms of the monoidal functor $\Id_\CC$. Moreover, $j
      \mapsto j_0\inv j$ defines a bijection between the set of pivotal structures
      of $\CC$ and $\Aut_\o(\Id_\CC)$. Note that $j$ is spherical if, and only if,
      the associated dimension function is real valued, and hence for any simple
      $V$, $(j_0\inv j)_V = \lambda_V \id_V$ for some real scalar $\lambda_V$.  By
      \cite[Thm. 6.2]{GN2}, $\Aut_\o(\Id_\CC) \cong G(\CC)$ and hence the first
      statement follows. In particular, $j_0\inv j$ has finite order. Thus, $j$ is
      a spherical structure of $\CC$ if, and only if, $(j_0\inv j)_V  = \pm
      \id_V$ for any simple $V$, or  $j_0\inv j \in \Aut_\o(\Id_\CC)$ is of order
      $\le2$. Therefore, the second statement follows from the isomorphism
      $\Aut_\o(\Id_\CC) \cong G(\CC)$.
    \end{proof}

    \begin{remark}
      The isomorphism $\Aut_\o(\Id_\CC) \cong G(\CC)$ is determined by the braiding $c$
      and the spherical structure $j_0$ of the modular category $(\CC,c,j_0)$. By
      \cite[Cor. 7.11]{M2}, $(\CC, c, j)$ is a modular category for all spherical
      structures $j$ of $\CC$, so that there are exactly $|G(\CC)|$ pivotal and
      $|\Omega_2 G(\CC)|$ spherical structures on the fusion category $\CC$.
    \end{remark}

    In any ribbon fusion category $\CC$ the associated ribbon structure, $\theta$,
    has finite order.
    This celebrated fact is part of Vafa's Theorem
    (see \cite{Va1, BKi}) in the case of modular categories.
    However, any ribbon category embeds in a modular category
    (via Drinfeld centers, see \cite{M2}) so the result holds generally. Observe
    that, $\theta_{V_i} = \th_i \id_{V_i}$ for some root of unity $\th_i \in \BC$.
    Since $\theta_\1 = \id_\1$, $\th_0 = 1$. The $T$-\defnfont{matrix} of
    $\CC$ is defined by $T_{ij} = \delta_{ij} \th_j$ for $i,j \in \Pi_\CC$. The
    \textbf{balancing equation}:
    \begin{equation}
      \label{Balancing}
      \theta_i\theta_j S_{ij}=\sum_{k\in\Pi_\CC} N_{i^*j}^kd_k\theta_k
    \end{equation}

    is a useful algebraic consequence, holding in any premodular category.
    The pair $(S,T)$ of $S$ and $T$-matrices will be called the \defnfont{modular
    data} of a given modular category $\CC$.

  \subsubsection{Modular Data and $\SL{}$ Representations}
   \begin{definition}
      For a pair of matrices $(S,T)$ for
      which there exists a modular category with modular data $(S,T)$, we will say
      $(S,T)$ is \defnfont{realizable modular data}.
    \end{definition}

    The fusion rules $\{N_{i,j}^k\}_{i,j,k \in \Pi_\CC}$ of
    $\CC$ can be written in terms of the $S$-matrix, via the
    \defnfont{Verlinde formula} \cite{BKi}:
    \begin{equation}
      \label{Verlinde Formula}
      N_{i,j}^k = \frac{1}{D^{2}} \sum_{a \in \Pi_\CC} \frac{S_{ia} S_{ja}
      S_{k^*a}}{S_{0a}} \quad\text{for all } i,j,k \in \Pi_\CC\,.
    \end{equation}

    The  modular data $(S, T)$ of a modular category $\CC$ satisfy the conditions:
    \begin{equation}
      \label{eq:STrelations}
      (ST)^3 = p^{+} S^2, \quad S^2=p^{+}p^{-} C, \quad CT=TC,\quad C^2=\id,
    \end{equation}

    where $p^{\pm} = \sum_{i\in \Pi_\CC} d_i^2 \th_i^{\pm 1} $ are called the
    \defnfont{Gauss sums}, and
    $C=\(\delta_{ij^*}\)_{i,j \in \Pi_\CC}$ is called the \defnfont{charge conjugation matrix} of $\CC$.
    In terms of matrix entries, the first equation in (\ref{eq:STrelations})
   gives the \textbf{twist equation}:
    \begin{equation}\label{twisteq}
      p^{+} S_{jk} = \theta_j \theta_k \sum_{i} \theta_i S_{ij} S_{ik}\,.
    \end{equation}
    The quotient $\frac{p^{+}}{p^{-}}$, called the \defnfont{anomaly} of $\CC$,
    is a  root of unity, and
    \begin{equation}
      \label{eq:gausssum}
      p^{+} p^{-}  =D^{2}.
    \end{equation}

    Moreover, $S$ satisfies
    \begin{equation}
      \label{eq:S}
      S_{ij} = S_{ji} \quad \text{and}\quad S_{ij^*} = S_{i^*j}
    \end{equation}

    for all $i, j \in \Pi_\CC$. These equations and the Verlinde formula imply that
    \begin{equation}
      \label{eq:orthogonality}
      S_{ij\du} = \ol{S_{ij}} \quad\text{and}\quad\frac{1}{D^{2}}\sum_{j \in \Pi_\CC}
      S_{ij}\ol{S}_{jk} = \delta_{ik}.
    \end{equation}

    In particular, $S$ is projectively unitary.

    A modular category $\CC$ is called \defnfont{self-dual} if $i=i\du$ for all $i \in
    \Pi_\CC$. In fact, $\mcC$ is self-dual if and only if $S$ is a real matrix.

    Let $D$ be the positive square
    root of $D^{2}$. The Verlinde formula can be rewritten as
    \begin{equation*}
      S N_i S\inv = D_i \quad \text{for }i \in \Pi_\CC
    \end{equation*}

    where $\(D_i\)_{ab} = \delta_{ab} \frac{S_{ia}}{S_{0a}}$. In particular, the
    assignments $\phi_a: i \mapsto \frac{S_{ia}}{S_{0a}}$ for $i \in \Pi_\CC$
    determine (complex) linear characters of $K_0(\CC)$. Since $S$ is non-singular,
    $\{\phi_a\}_{a \in \Pi_\CC}$ is the set of \textit{all} the linear characters of
    $K_0(\CC)$. Observe that $\FPdim$ is a character of $K_0(\CC)$, so that there is
    some $a\in\Pi_\CC$ such that $\FPdim=\phi_a$.  By the unitarity of $S$, we have
    that $\FPdim(\CC)=D^{2}/(d_a)^2$.

    As an abstract group, $\SL{} \cong \langle \fs, \ft \mid \fs^4=1, (\fs \ft)^3
    =\fs^2\rangle$. The standard choice for generators is:
    \begin{equation*}
      \fs := \mtx{0 & -1\\ 1 & 0}\quad \text{and} \quad \ft:=\mtx{1 & 1\\ 0 & 1}\,.
    \end{equation*}

    Let $\eta:\GLC{\Pi_\CC}\to
    \PGLC{\Pi_\CC}$ be the natural surjection. The relations \eqref{eq:STrelations} imply that
    \begin{equation}
      \label{eq:projrep}
      \orho_\CC\colon \fs\mapsto \eta(S)\quad \text{and}\quad \ft\mapsto \eta(T)
    \end{equation}

    defines a projective representation of $\SL{}$.  Since the  modular
    data is an invariant of a modular category, so is the associated projective
    representation type of $\SL{}$.  The following arithmetic properties of this
    projective representation will play an important role in our discussion (cf.
    \cite{NS3}).  Recall that $\BQ_N:=\BQ(\zeta_N)$, where $\zeta_N$ is a primitive $N$th root of unity.
    \begin{thm}
      \label{t:cong1}
      Let $(S,T)$ be the  modular data of the modular category $\CC$ with
      $N=\ord\(T\)$. Then the entries of $S$ are algebraic integers of $\BQ_N$.
      Moreover, $N$ is minimal such that the projective representation $\orho_\CC$ of
      $\SL{}$ associated with the modular data can be factored through $\mathrm{SL}(2,\BZ/N\BZ)$.
      In other words, $\ker\orho_\CC$ is a congruence subgroup of level $N$.
    \end{thm}

    \begin{definition}
      A \defnfont{modular representation} of $\CC$ (cf. \cite{NS3}) is a representation
      $\rho$ of  $\SL{}$ which satisfies the commutative diagram:
      \begin{equation*}
        \xymatrix{ \SL{}\ar[r]^-{\rho}\ar[rd]_-{\orho_\CC}
        &\GLC{\Pi_\CC}\ar[d]^-{\eta}\\
        & \PGLC{\Pi_\CC}\,.
        }
      \end{equation*}
    \end{definition}

    Let $\zeta\in \mbbC$ be a fixed $6$-th root of the anomaly
    $\dfrac{p^{+}}{p^{-}}$. For any $12$-th root of unity $x$, it follows from
    \eqref{eq:STrelations} that  the assignments
    \begin{equation}
      \label{eq:repC}
      \rho_x^\zeta: \fs \mapsto \frac{\zeta^3}{x^3 p^{+}} S, \quad \ft \mapsto \frac{x}{\zeta} T
    \end{equation}

    define a  modular representation of $\CC$.  Moreover, $\{\rho_x^\zeta\mid
    x^{12}=1\}$ is the complete set of modular representations of $\CC$ (cf.
    \cite[Sect. 1.3]{DLN1}).  Since $D^2=p^{+}p^{-}$, we have $\zeta^3/p^{+}=\gamma/D$, where $\gamma =\pm 1$. Thus, one can always find a $6$-th root of unity $x$ so that $\rho_x^\zeta: \fs \mapsto S/D$. For the purpose of this paper, we only need to consider the modular representation $\rho$ of $\CC$ which assigns $\fs \to S/D$.    Note also that $\rho_x^\zeta(\fs)$ and $\rho_x^\zeta(\ft)$
    are matrices over a finite abelian extension of $\BQ$. Therefore, modular
    representations of any modular category are defined over the abelian closure
    $\BQA$ of $\BQ$ in $\BC$ (cf. \cite{BG}).

    Let $\rho$ be any modular representation of the modular category $\CC$, and set
    \begin{equation*}
      s = \rho(\fs) \quad \text{and}\quad t = \rho(\ft)\,.
    \end{equation*}

    It is clear that a representation $\rho$ is uniquely determined by the pair
    $(s,t)$, which will be called a \defnfont{normalized modular pair} of $\CC$. In view of the preceding paragraph, there exists a root of unity $y$ such that $(S/D, T/y)$ is a  normalized modular pair of $\CC$.

  \subsubsection{Galois Symmetry}\label{galsymsection}

    Observe that for any choice of a normalized modular pair $(s,t)$, we have
    $\frac{s_{ia}}{s_{0a}}=\frac{S_{ia}}{S_{0a}}=\phi_a(i)$.
    For each $\s \in \Aut(\BQA)$, $\s(\phi_a)$  given by $\s(\phi_a)(i) =
    \s\left(\frac{s_{ia}}{s_{0a}}\right)$ is again a linear character of $K_0(\CC)$
    and hence $\s(\phi_a) = \phi_{\hs(a)}$ for some unique $\hs\in \Sym(\Pi_\CC)$.
    That is,
    \begin{equation}
      \label{eq:galois1}
      \s\left(\frac{s_{ia}}{s_{0a}}\right) = \frac{s_{i\hs(a)}}{s_{0\hs(a)}} \quad \text{for all }i, a\in
      \Pi_\CC\,.
    \end{equation}
    Moreover, there exists a function $\e_\s : \Pi_\CC \to \{\pm 1\}$, which
    depends on the choice of $s$, such that:
    \begin{equation}
      \label{eq:galois2}
     \s(s_{ij}) = \e_{\s}(i) s_{\hs(i) j} = \e_{\s}(j) s_{i \hs(j)} \quad \text{for all }i, j \in \Pi_\CC
    \end{equation}
    (cf.  \cite[App. B]{BG}, \cite{CG1} or \cite[App.]{ENO1}). The group $\Sym(\Pi_\CC)$ will often be written as $\mathfrak{S}_r$ where $r=|\Pi_\CC|$ is the rank of $\CC$.

    The following theorem will be used in the sequel:
    \begin{thm} \label{t:Galois}
      \label{c:galsym}
      Let $\CC$ be a modular  category of rank $r$, with $T$-matrix of order $N$. Suppose $(s,t)$ is a normalized modular pair of $\CC$. Set $t=(\delta_{ij} t_i)$ and $n =\ord (t)$. Then:
      \begin{enumerate}
       \item[(a)] $N \mid n\mid 12 N$ and
       $s,t \in \GL_r(\BQ_n)$.  Moreover,
       \item[(b)] (Galois Symmetry) for $\s\in \Gal(\BQ_n/\BQ)$,
       $
  \s^2(t_i)=t_{\hs(i)}.
  $
      \end{enumerate}
    \end{thm}
Part $(a)$ of Theorem \ref{t:Galois} is proved in \cite{NS3}, whereas
     part $(b)$ is proved in \cite[Thm. II(iii)]{DLN1}.

In the sequel, we will simply denote by $\BF_A$
    the field extension over $\BQ$ generated by the entries of a complex matrix
    $A$. If $\BF_A/\BQ$ is Galois, then we simply write $\Gal(A)$ for the Galois
    group $\Gal(\BF_A/\BQ)$.

In this notation, if $(S,T)$ is the modular data of $\CC$, then $\BF_T = \BQ_N$,
    where $N=\ord\(T\)$, and we have  $\BF_S \subseteq \BF_T$ by \thmref{t:cong1}. In
    particular, $\BF_S$ is an abelian Galois extension over $\BQ$.

For any normalized modular pair $(s,t)$ of $\CC$ we have $\BF_t = \BQ_n$, where
    $n=\ord\(t\)$. Moreover, by  \thmref{t:Galois}, $\BF_S\subseteq \BF_s \subseteq \BF_t$. In
    particular, the field extension $\BF_s/\BQ$ is also Galois. The kernel of the
    restriction map $\res:\Gal(t) \to \Gal(S)$ is isomorphic to $\Gal(\BF_t/\BF_S)$.

    The following important lemma is proved in \cite[Prop. 6.5]{DLN1}.
    \begin{lem}
      \label{l:2group}
      Let $\CC$ be a modular category with modular data $(S,T)$. For any
      normalized modular pair $(s,t)$ of $\CC$, $\Gal(\BF_t/\BF_S)$ is an
      elementary 2-group.
    \end{lem}

  \subsubsection{Frobenius-Schur Indicators}\label{fs indicators}

    A \defnfont{strict pivotal} category is a pivotal category in which the
    associativity isomorphisms are identities,  the pivotal structure $j:\Id_\CC\rightarrow
    (-)\bidu$ is the identity, and
    the associated natural isomorphisms $\xi_{U, V}: U^{*} \o V^{*} \to (V \o
    U)^{*}$ are also identities. Moreover, we have the following theorem (cf.
    \cite{NS1}).
    \begin{thm}
      \label{t:str_piv}
      Every pivotal category is pivotally equivalent to a strict pivotal category.
    \end{thm}

    Frobenius-Schur indicators are indispensable invariants of spherical categories
    introduced in \cite{NS1}.  They are defined for each object in a pivotal
    category. Here, we only provide the definition of these indicators in a
    \emph{strict} spherical category. Let $n$ be a positive integer and $V$ an object of a strict spherical category
    $\CC$. We denote by $V^{\ot n}$ the $n$-fold tensor power of $V$. One can define a
    $\BC$-linear operator $E_V^{(n)}: \Hom_{\mcC}\(\1, V^{\ot n}\) \to
    \Hom_{\mcC}\(\1, V^{\ot n}\)$ given by
    \begin{equation*}
      E_V^{(n)}(f) = \left( \1 \xrightarrow{\db} V^{*} \o V \xrightarrow{\id_{V^{*}}
      \o f\o \id_{V}} V^{*} \o V^{\ot n+1} \xrightarrow{\ev \o \id_{V}^{\ot n}} V^{\ot n}
      \right)\,.
    \end{equation*}

    The $n$-th Frobenius-Schur indicator of $V$ is defined as
    \begin{equation*}
      \nu_n(V) = \Tr(E_V^{(n)})\,.
    \end{equation*}

    It follows directly from graphical calculus that $\left(E_V^{(n)}\right)^n
    =\id$, and so $\nu_n(V)$ is an algebraic integer in the $n$-th cyclotomic
    field $\BQ_n = \BQ(e^{\frac{2\pi i}{n}})$.

    The first indicator $\nu_1(V_i)$ is the Kronecker delta function $\delta_{0i}$ on
    $\Pi_\CC$, i.e. $\nu_1(V)=1$ if $V \cong \1$ and 0 otherwise. The second
    indicator is consistent with the classical Frobenius-Schur indicator of an
    irreducible representation of a group, namely $\nu_2(V)=\pm 1$ if $V \cong
    V^{*}$ and 0 otherwise for any simple object $V$ of $\CC$. The higher indicators
    are more obscure in nature, but they are all additive complex valued functions of
    the Grothendieck ring $K_0(\CC)$ of $\CC$.

    The classical definition of exponent
    of a finite group can be generalized to a spherical category via the following
    theorem \cite{NS2}.
    \begin{thm}
      Let $\CC$ be a spherical  category. There exists a positive integer $n$ such
      that $\nu_n(V) = d\(V\)$ for all $V \in \CC$. If $N$ is minimal among such $n$,
      then $d\(V\)$ are algebraic integers in $\BQ_{N}$.
    \end{thm}

    The minimal integer $\FSexp(\CC):=N$ above is called the
    \defnfont{Frobenius-Schur exponent}. If $\CC$ is the category of complex
    representations of a finite group $G$, then $\FSexp(\CC) = \exp(G)$.
For modular categories the
      Frobenius-Schur indicators $\nu_n(V)$ are completely
      determined
      by the modular data of $\CC$, explicitly given in \cite{NS2} (generalizing the second
      indicator formula in \cite{Ban1}):
    \begin{thm}
      Let $\CC$ be a modular category with the $T$-matrix given by
      $[\delta_{ij}\theta_i]_{i,j \in \Pi_\CC}$. Then $\ord\(T\)=\FSexp(\CC)$, and
      \begin{equation}
        \label{eq:BF}
        \nu_n(V_k) = \frac{1}{D^{2}} \sum_{i,j \in \Pi_\CC} N_{i,j}^k\, d_i
        d_j\left(\frac{\theta_i}{\theta_j}\right)^n
      \end{equation}

      for all $k \in \Pi_\CC$ and positive integers $n$.
    \end{thm}
\subsubsection{Modular Data}\label{modulardata}
    \begin{definition}
      Let $S,T\in\GL_r(\BC)$ and define constants $d_j:=S_{0j}$, $\theta_j:=T_{jj}$, $D^2:=\sum_j d_j^2$
      and $p_{\pm}=\sum_{k=0}^{r-1}(S_{0,k})^2\th_{k}^{\pm1}$.  The pair $(S,T)$
      is an \defnfont{admissible modular data} of rank
      $r$ if they satisfy the following conditions:
      \begin{enumerate}
        \item $d_j\in\BR$ and $S=S^t$ with $S\overline{S}^t=D^2 \Id$.
          $T_{i,j}=\delta_{i,j}\theta_i$ with $N:=\ord(T)< \infty$.
        \item $(ST)^3=p^{+}S^2$, $p_{+}p_{-}=D^2$ and $\frac{p_{+}}{p_{-}}$ is a root of unity.
        \item $N_{i,j}^k:=\frac{1}{D^2} \sum_{a=0}^{r-1} \frac{S_{ia} S_{ja}
          \overline{S_{ka}}}{S_{0a}}\in\BN$ for all $0\leq i,j,k\leq (r-1)$.
        \item   $\theta_i\theta_j
          S_{ij}=\sum_{k=0}^{r-1} N_{i^*j}^kd_k\theta_k$ where $i^*$ is the unique label such that $N_{i,i^*}^0=1$.
        \item Define $\nu_n(k): = \frac{1}{D^2} \sum_{i,j =0}^{r-1} N_{i,j}^k\,
          d_i d_j\left(\frac{\theta_i}{\theta_j}\right)^n$. Then
          $\nu_2(k)=0$ if $k\neq k^*$ and $\nu_2(k)=\pm 1$ if $k=k^*$.  Moreover,
          $\nu_n(k)\in \BZ[e^{2 \pi i/N}]$ for all $n,k$.
        \item $\mbbF_S\subset \BF_T=\BQ_N$, $\Gal(\mbbF_S/\BQ)$ is isomorphic to an
          abelian subgroup of $\mathfrak{S}_r$ and $\Gal(\mbbF_T/\mbbF_S)\cong (\BZ/2\BZ)^\ell$ for some integer $\ell$.
        \item The prime divisors of $D^{2}$ and $N$ coincide in
        $\BZ[e^{2 \pi i/N}]$.\footnote{See \sectionref{subsection: The Cauchy Theorem}.}
      \end{enumerate}
    \end{definition}
%

    \begin{theorem}
      \label{realizeadmiss}
      Let $\(S,T\)$ be a realizable modular data. Then
      \begin{itemize}
        \item[(a)] $\(S,T\)$ is admissible and
        \item[(b)] For all $\s\in\Aut(\BQA)$,
        $\(\s\(S\),\s\(T\)\)$ is realizable.
      \end{itemize}
    \end{theorem}
    \begin{proof}
      (a) follows from the definition of admissible modular data, while (b)
      follows from \cite[Section 2.7]{ENO1} (see also \cite{DHW1}).
    \end{proof}

    \begin{remark}
      We expect a converse of Theorem \ref{realizeadmiss} to be true: that is, if
      $(S,T)$ is admissible then it is realizable.  Indeed, a satisfactory definition
      of admissible would be a minimal set of conditions that guarantee realizability.
    \end{remark}

\section{Rank-Finiteness and the Cauchy Theorem}
  \label{Rank Finiteness and the Cauchy Theorem}
  The main goal of this section is to prove the following
  theorem, conjectured by the fourth author in 2003 (see \cite{W3}):

  \begin{theorem}[Rank-Finiteness Theorem]
    \label{Rank Finiteness}
    There are only finitely many modular categories of fixed rank $r$, up to
    equivalence.
  \end{theorem}

  Prior to this work this conjecture had only been resolved in certain restricted
  cases, for instance it was shown \cite[Proposition 8.38]{ENO1} that there are
  finitely many \emph{weakly integral fusion} categories of a given fixed rank through a  number
  theoretic argument similar to that of Landau \cite{L1}.

  The proof of the Rank-Finiteness Theorem relies upon several well-known reductions, a new result known as the
  Cauchy Theorem (for Spherical Fusion Categories \ref{Cauchy Theorem for Modular
  Categories}) and some results in analytic number theory due to Evertse \cite{Ev1}.

  In \sectionref{subsection: The Cauchy Theorem}, the Cauchy Theorem for Spherical Fusion
  Categories is proved, and in \sectionref{subsection:
  Rank Finiteness} we prove Theorem \ref{Rank Finiteness}.  We discuss asymptotics related to Theorem \ref{Rank Finiteness} in \sectionref{asymptotics}.

  \subsection{The Cauchy Theorem}
  \label{subsection: The Cauchy Theorem}

Let $\mathbb{A}$ be the ring of algebraic integers in $\mbbC$. For $a, b, c \in
\mathbb{A}$ with $a \ne 0$, $b \equiv c \bmod a$ means that $(b-c)/a \in
\mathbb{A}$.

Suppose $\CC$ is a modular category with $N=\FSexp(\CC)$ and $q$ is prime with
$(q,N)=1$.
We begin with a simple lemma, which is essentially proved in \cite[Lem. 1.8]{Wa}
and \cite[Section 3.4]{KSZ1}.
\begin{lem}\label{l1}
  Let $W$ be a finite-dimensional $\mbbC$-linear space. If $E$ is a $\mbbC$-linear operator on $W$ such that $E^q =\id_W$ for some prime number $q$, then
  $$\Tr(E)^q  \equiv \dim_\mbbC W \bmod q\,.$$
  In particular, if $\Tr(E) \in \BZ$, then $\Tr(E) \equiv \dim_\mbbC W \bmod q$.
\end{lem}
\begin{proof}
  Let $\zeta_q\in \mbbC$ denote a primitive $q$-th root of unity.
  Then $\Tr(E)= \sum_{i=0}^{q-1} m_i \zeta_q^i$, where $m_i$ is the multiplicity of the eigenvalue $\zeta_q^i$. Thus,
  \begin{equation*}
    \Tr(E)^q \equiv \sum_{i=0}^{q-1} m_i^q \equiv \sum_{i=0}^{q-1} m_i =
    \dim_\mbbC W \bmod q\,.
  \end{equation*}

  In particular, if $\Tr(E) \in \BZ$,  the second statement follows from  Fermat's little theorem.
\end{proof}

Recall that  the $n$-th Frobenius-Schur indicator $\nu_n(V)$ for $V \in \CC$ is
defined as the trace of a $\mbbC$-linear operator $E^{(n)}_{V} : \Hom_{\mcC}\(\1,
V^{\ot n}\) \to \Hom_{\mcC}\(\1, V^{\ot n}\)$.  This operator $E_V^{(n)}$ satisfies
\begin{equation*}
  \left(E_V^{(n)}\right)^n= \id\,.
\end{equation*}

Moreover, $\nu_n(V)$ is an algebraic integer in $\BQ_n\cap\BQ_N$. Since $q$ and
$N$ are relatively prime, we have
\begin{equation*}
  \nu_q(V) \in \BQ_N \cap \BQ_q =\BQ\,.
\end{equation*}

Thus $\nu_q(V) \in \BZ$. By the preceding lemma, we have proved
\begin{lem} \label{l2}
  For any $V \in \CC$, $\nu_q(V) \in \BZ$ and we have
  \begin{equation*}
    \nu_q(V) \equiv \dim_\mbbC \Hom_{\mcC}\(\1, V^{\ot q}\) \bmod q\,.
  \end{equation*}
\end{lem}

Let $\OO_N$ be the ring of algebraic integers of $\BQ_N$. It is well known that
$\OO_N=\BZ[\zeta_N]$, where $\zeta_N$ is a primitive $N$-th root of unity in
$\mbbC$. Set $K_N(\CC) = K_0(\CC) \o_\BZ \OO_N$. Then $K_N(\CC)$ is an
$\OO_N$-algebra. For any non-zero element $a \in \OO_N$ and $\a, \b \in
K_N(\CC)$, we write $\a \equiv \b \bmod a$ if $\a - \b= a \g$ for some $\g
\in K_N(\CC)$.

By \cite{NS2}, $\nu_q: K_0(\CC) \to \BZ$ is a group homomorphism; however, the
assignment $V \mapsto \dim_\mbbC \Hom_{\mcC}\(\1, V^{\ot q}\) $ is not. We can extend the $\nu_q$
to an $\OO_N$-linear map from $K_N(\CC)$ to $\OO_N$, and we continue to denote
such an extension by $\nu_q$. Similarly, we can extend the dimension function
$d: K_0(\CC) \to \OO_N$ to an $\OO_N$-linear map from $K_N(\CC)$ to $\OO_N$.
However, it is important to note that this  extension is an
$\OO_N$-\emph{algebra} homomorphism.

Note that $K_N(\CC)$ is a free $\OO_N$-module with $\Pi_\CC$ as a basis. For
$\a =
\sum_{i \in \Pi_\CC}\a_i i \in K_N(\CC)$, we define $\d(\a) = \a_0$. Obviously,
$\d: K_N (\CC)\to \OO_N$ is $\OO_N$-linear. Although $\d(\a^q)$ is not
$\OO_N$-linear in $\a$, it is $\BZ$-linear modulo $q$.
\begin{lem}\label{l3}
  For $\a \in K_N(\CC)$,  we have
  \begin{equation*}
    \d(\a^q) \equiv \sigma_q(\nu_q(\a)) \bmod q
  \end{equation*}
  where $\s_q \in \Gal(\BQ_N/\BQ)$ is defined by $\s_q(\zeta_N)= \zeta_N^q$.
\end{lem}
\begin{proof}
  Let $\a = \sum_{i \in \Pi_\CC} \a_i i$. Then $\a^q \equiv \sum_{i \in \Pi_\CC}
\a_i^q i^q \bmod q$. Since $\d$ is $\OO_N$-linear, we have
  \begin{equation*}
    \d(\a^q) \equiv \d(\sum_{i \in \Pi_\CC} \a_i^q i^q) =\sum_{i \in \Pi_\CC}
\a_i^q \d(i^q)\bmod q\,.
  \end{equation*}

  Since $\d(i^q) = \dim_\mbbC \Hom_{\mcC}(\1,V_i^{\ot q})$, it follows from Lemma \ref{l2} that
  \begin{equation*}
    \d(i^q)  \equiv \nu_q(V) \bmod q\,.
  \end{equation*}

  Thus, we find
  \begin{equation*}
    \d(\a^q) \equiv \sum_{i \in \Pi_\CC} \a_i^q \nu_q(V_i)  \bmod q\,.
  \end{equation*}

  Note that for $a \in \OO_N$, $a = \sum_j a_j \zeta_N^j$ where $a_j \in \BZ$. Therefore,
  \begin{equation*}
    a^q \equiv \sum_j a_j^q \zeta_N^{qj} \equiv \sum_j a_j \sigma_q (\zeta_N^j) = \sigma_q(a) \bmod q\,.
  \end{equation*}

  Hence, we have
  \begin{equation*}
    \d(\a^q) \equiv \sum_{i \in \Pi_\CC} \sigma_q(\a_i) \nu_q(V_i)
=\sigma_q(\nu_q(\a)) \bmod q\,.
  \end{equation*}

  The last equality follows from the $\OO_N$-linearity of $\nu_q$, and
$\nu_q(V_i) \in \BZ$ for all $i \in \Pi_\CC$.
\end{proof}

By \cite{NS2}, $d_i \in \OO_N$ for $i \in \Pi_\CC$. Therefore, $R=\sum_{i \in
\Pi_\CC}
d_i i$ is an  element of $K_N(\CC)$.

Notice that $R$ defines a rank 1 ideal of $K_N(\CC)$ as $i R = d_i R$ for all
$i \in \Pi_\CC$. Thus, for $\a \in K_N(\CC)$, $\a R = d(\a) R$, where $d:
K_N(\CC) \to \OO_N$ is the extended dimension function.
 Therefore,
\begin{equation*}
  R^n =  R^{n-1} R = d(R^{n-1})  R = D^{2(n-1)} R\,.
\end{equation*}

Now, we can write our first proposition for the indicators of the \textit{virtual} object $R$.
\begin{prop}\label{p1}
  Let $R=\sum_{i \in \Pi_\CC} d_i i \in K_N(\CC)$, and $q$ a prime number not
dividing
  $N$. Then we have
  \begin{equation*}
    \s_q\left(\nu_q(R)\right) \equiv D^{2(q-1)} \bmod q\,.
  \end{equation*}
\end{prop}
\begin{proof}
  By the preceding discussion, we have $R^q = D^{2(q-1)} R$. Since $\d(R)=d_0 =1$,
  we have $\delta(R^q) = D^{2(q-1)}$. By Lemma \ref{l3},
  \begin{equation*}
    \sigma_q(\nu_q(R)) \equiv D^{2(q-1)}  \bmod q\,. \qedhere
  \end{equation*}
\end{proof}

\begin{prop}
  \label{p3}
  For any $\s \in \Gal(\BQ_N/\BQ)$, $d_{\hs(0)}$ is a unit of $\OO_N$.
\end{prop}
\begin{proof}
  Without loss of generality, we may assume  $\hs(0)=1$. Then
  \begin{equation*}
    \s(\frac{1}{D^2}) = d^2_1/D^2
    \quad \text{or}\quad d_1^2 = D^2/\s(D^2).
  \end{equation*}
  Obviously, the norm of $D^2/\s(D^2)$ is 1, and so is $d_1^2$. Therefore $d_1$ is a unit in $\OO_N$.
\end{proof}

\begin{remark}
  The preceding is another proof of the fact proved in \cite{O3} that $D^2$ is a
  $d$-number.
\end{remark}

\begin{prop}\label{p2}
  $\nu_q(R) = d_{\hs_q\inv(0)}^2$\,.
\end{prop}
\begin{proof}
  Note that
  \begin{equation*}
    \nu_q(R) = \sum_{k} d_k \nu_q(V_k)
             = \frac{1}{D^2} \sum_{i,j,k} d_k N_{ij}^k d_i d_j \frac{\theta_i^q}{\theta^q_j}
             = \frac{1}{D^2} \left(\sum_i d_i^2 \theta^q_{i}\right)\left(\sum_j d_j^2 \theta^{-q}_{j}\right)\,.
  \end{equation*}
Choose $y$ so that $s=S/D$ and $t=T/y$ give a modular representation.
  Reexpressing \eqnref{eq:gausssum}, we find that:
  %
  \begin{equation*}
    s_{00}^2 =\frac{1}{D^2}= \left(\sum_i s_{0i}^2 \theta_i\right)\left(\sum_j s_{0j}^2 \theta_j\inv\right)\,.
  \end{equation*}

   Let $\tau \in \Aut(\BQA)$ such that $\tau|_{\BQ_N} = \s_q\inv$.
  Applying $\tau$ to this equation and applying (\ref{eq:galois2}), we find
  \begin{equation*}
    s_{\htau(0)0}^2=d_{\htau(0)}^2 s_{00}^2 = \left(\sum_i s_{0\htau(i)}^2 \t(\theta_i)\right)\left(\sum_j s_{0\htau(j)}^2 \t(\theta_j)\inv\right)\,.
  \end{equation*}

  By Galois symmetry, \thmref{t:Galois}, we have
  $$
  \t(t_i)=\t^{-1}\t^2(t_i)=\tau\inv\left(t_{\htau(i)}\right) = \frac{\theta^q_{\htau(i)}}{\tau\inv(y)}
  $$
   since $\tau\inv|_{\BQ_N}=\s_q$.  Therefore, in terms of
  $\theta_i=yt_i$, we get
  $\t(\theta_i^{\pm 1})=\theta_{\htau(i)}^{\pm q}K^{\pm 1}$ for some constant $K$ independent of $i$, which yields:

  \begin{equation*}
    d_{\htau(0)}^2 s_{00}^2 = \left(\sum_i s_{0\htau(i)}^2 \theta^q_{\htau(i)}\right)\left(\sum_j s_{0\htau(j)}^2 \theta^{-q}_{\htau(j)}\right)\,.
  \end{equation*}

  Thus,
  \begin{equation*}
    d_{\htau(0)}^2 = \frac{1}{D^2} \left(\sum_i d_i^2 \theta^q_{i}\right)\left(\sum_j d_j^2 \theta^{-q}_{j}\right)=\nu_q(R)\,. \qedhere
  \end{equation*}
\end{proof}

\begin{thm}[Cauchy Theorem for Spherical Fusion Categories]
  \label{Cauchy Theorem for Modular Categories}
  Let $\CC$ be a spherical fusion category. Then set of prime ideals dividing the principal ideal generated by $D^2$  is identical to that of $N=\FSexp(\CC)$ in $\OO_N=\mathbb{Z}[e^{2\pi i/N}]$.
\end{thm}
\begin{proof} We first consider the case when $\CC$ is a modular category. Since $N \mid D^6$ \cite{Eti1}, every prime ideal factor of $N$ in $\OO_N$ is a factor of $D^2$. Suppose $\mathfrak{p}$ is a prime ideal factor of $D^2$.
  By Propositions \ref{p1} and \ref{p2}, we find the congruence
  \begin{equation*}
    \s_q(d_i^2) \equiv D^{2(q-1)} \bmod q
  \end{equation*}

  for any prime $q \nmid N$ where $i= \hs_q\inv(0)$.
  By Proposition \ref{p3}, $d_i$ is a unit of $\OO_N$. Therefore, we have the equality of ideals $\OO_N
  = (D^2)+(q)$. This implies that $q \not\in \mathfrak{p}$. Therefore,
  $\mathfrak{p}\cap \BZ = (p)$ for some prime $p\mid N$. Hence, $\mathfrak{p}$
  is a prime factor of $N$ in $\OO_N$.

  Now, we assume $\CC$ is a general spherical fusion category. Then its Drinfeld center $Z(\CC)$  is modular and $\dim Z(\CC)=(\dim \CC)^2$ by \cite{M2}. Since  $N=\FSexp(\CC)=\FSexp(Z(\CC))$ by \cite{NS3}, the theorem follows from the modular case \emph{i.e.} $N$ and $(\dim \CC)^2$ have the same set of prime ideal factors in $\OO_N$.
\end{proof}

\begin{rmk}
If $\CC$ is the category of representations of a finite group $G$, then $\FSexp(\CC)= \exp(G)$ and $\dim \CC =|G|$. The preceding theorem implies that $p$ is a prime factor of $|G|$ if, and only if, $p \mid \exp(G)$; this is simply an equivalent statement of the classical Cauchy Theorem for finite groups.
\end{rmk}

  \subsection{Proof of Rank-Finiteness}
  \label{subsection: Rank Finiteness}

To prove Theorem \ref{Rank Finiteness}, we first reduce to proving that there are finitely many
possible fusion rules using (braided)
 Ocneanu Rigidity, due to Ocneanu, Blanchard and Wassermann (unpublished), a proof of which may be found in \cite{ENO1}:
\begin{theorem}
  There are only finitely many (braided, modular) fusion categories which have
  the same fusion rules up to (braided, modular) monoidal equivalence.
\end{theorem}

\begin{rmk}
  Ocneanu Rigidity for fusion categories was first proved by Blanchard and Wasserman, and the extension
  to the braided case can be found in
  \cite[Remark 2.33]{ENO1}.  For the finiteness of spherical structures see
 \lemmaref{pivotalinvertible}.
\end{rmk}

Next we may reduce to bounding the FP-dimension using (see eg.
\cite[Proposition 8.38]{ENO1} and \cite[Proposition 6.2]{RSW}):
\begin{cor}
  \label{201302051239}
  There are finitely many (braided, modular) fusion categories $\mcC$ satisfying
  $\FPdim\(\mcC\)\leq M$ for any fixed number $M>0$, up to (braided, modular)
  monoidal equivalence.
\end{cor}

For the reader's convenience, we provide an explicit bound on the $N_{ij}^k$ in
terms of $\FPdim(\CC)$ for fusion categories.
\begin{lemma}
  \label{201302051302}
  If $\mcC$ is a rank $n$ fusion category, then for $a \in \Pi_\CC$, we have the inequality:
  \begin{align*}
    \| N_{a}\|_{max}\leq \FPdim\(V_{a}\)\leq
    n\| N_{a}\|_{max}
  \end{align*}
  where  $\|A \| _{max}$ is the max-norm of the complex matrix $A$
  given by
  \begin{align*}
    \| A \|_{max}=\max_{i,j} \abs{A_{ij}}\,.
  \end{align*}
\end{lemma}
\begin{proof}
Note that
$$
R_0 = \sum_{a \in \Pi_\CC} \FPdim(a) a
$$
 generates a 1-dimensional ideal of $K_\BC(\CC) = K_0(\CC) \o_\BZ \BC$, and that $a R_0 = \FPdim(V_a) R_0$ for all $a \in \Pi_\CC$. In particular, there is unit vector $x$ with positive components such that $N_a x = \FPdim(V_a) x$ for all $a \in \Pi_\CC$.

 Let $\r\(A\)$ denote the spectral radius of an $n\times n$ complex matrix $A$. Recall that the 2-norm of $A$ is given by $\| A \|_2 =\sqrt{\r\(A^* A\)}$.  Thus, for $a \in \Pi_\CC$, $\| N_a \|_2 \ge \FPdim(V_a)$. On the other hand,
\begin{multline*}
\| N_a \|_2^2 = \r(N_a^* N_a) = \r(N_{a^*} N_a) = \r\left(\sum_{b \in \Pi_\CC} N_{a^*,a}^b N_b  \right) \\
\le \sum_{b \in \Pi_\CC} N_{a^*,a}^b \FPdim(V_b) = \FPdim(V_a^*)\FPdim(V_a) = \FPdim(V_a)^2\,.
\end{multline*}
Therefore, $\| N_a \|_2 = \FPdim(V_a)$ for all $a \in \Pi_\CC$. The result then follows by the inequality
 $$
  \|A\|_{max}\le \|A\|_{2}\leq n \|A\|_{max}
  $$
   for any $n\times n$
  complex matrix $A$.
\end{proof}

 Next we give an explicit bound on $\ord\(T\)$ in terms of the rank of $\CC$.
\begin{prop}
\label{201301161421}
  If $\mcC$ is a modular category of rank $r$ with modular data
  $\(S,T\)$, then $\ord\(T\)\leq 2^{2r/3+8}3^{2r/3}$.
\end{prop}
\begin{proof}
  By \cite{BG1}, any abelian subgroup $G$ of
  $\mfS_{r}$ satisfies $\abs{G}\leq 3^{r/3}$.
  On the other hand,
  since $\Gal\(\BF_T/\mbbQ\)\cong(\BZ/N\BZ)^*$, by \lemmaref{l:2group}, $\[\BF_T:\BF_S\]\leq 2^{m}$ where $m-1$ is the number of prime
  factors of $\ord\(T\)$. The Fundamental Theorem of Galois Theory can be
  utilized to relate $m$ and $[\BF_S:\mbbQ]$. To do this, we
  note that:
  $$\Gal\(\BF_S/\mbbQ\)\cong\Gal\(\BF_T/\mbbQ\)/\Gal\(\BF_T/\BF_S\)$$
  where $$\Gal\(\BF_T/\mbbQ\)\cong(\BZ/N\BZ)^*\cong \prod_i (\BZ/p_i^{\alpha_i}\BZ)^*$$ in terms of distinct prime-power cyclic subgroups.
  Since
  $\Gal\(\BF_T/\BF_S\)$ is an elementary $2$-group by \lemmaref{l:2group}, we see that at least $m-3$ (non-trivial)
  cyclic factors survive in the quotient (the three possible exceptions correspond to
  primes $2$ and $3$ in $\ord\(T\)$.) The structure of the maximal abelian
  subgroup of $\mfS_{r}$ ensures that $m-3\leq r/3+1$. 
  %
  %
  It follows that:
  \begin{align*}
    \[\BF_T:\mbbQ\]=\[\BF_T:\BF_S\]\[\BF_S:\mbbQ\]\leq
    2^{m}3^{r/3}\leq 2^{r/3+4}3^{r/3}\,.
  \end{align*}

  On the other hand, $\BF_T=\mbbQ_{\ord\(T\)}$ and so
  $\[\BF_T:\mbbQ\]=\vph\(\ord\(T\)\)$. In particular, if $\ord\(T\)\neq2$
  or $6$, then $\[\BF_T:\mbbQ\]\geq\sqrt{\ord\(T\)}$.
  Thus $\ord\(T\)\leq 2^{2r/3+8}3^{2r/3}$ since $2^{2/3+8}3^{2/3}>6$.
\end{proof}

The last ingredient of the proof of Theorem \ref{Rank Finiteness} is a deep result from analytic number theory, which necessitates some further notation and background.

\begin{defn}
  Let $\mbbK$ be a number field and $\mcS$ be a finite set of prime ideals in
  the ring $\mcO_{\mbbK}$ of algebraic integers of $\mbbK$. An element $\a\in\mbbK^{\times}$ is a \defnfont{$\mcS$-unit} if
  the prime factors of the principal fractional ideal $\ideal{\a}$ are all
  in $\mcS$.
\end{defn}

\begin{rmk}
  The $\mcS$-units form a finitely generated multiplicative abelian group which
  we will denote by $\mcO_{\mbbK,\mcS}^{\times}$ \cite{We1}.
\end{rmk}

\begin{rmk}
  It should be noted that $\mcS$-units are often treated adelically in which case
  a more delicate treatment involving places is required.  While we will not
  need this level of detail here, it should be mentioned that it is utilized in
  \cite{Ev1}. A detailed introduction to $\mcS$-units and their relationship to
  adeles can be found in most modern texts on advanced number theory e.g.
  \cite{We1}.
  \end{rmk}

The $\mcS$-units arise in a wide range of subdisciplines in number theory and
are typically found to obey an \defnfont{$\mcS$-unit equation}:
\begin{align*}
  x_{0}+\cdots+x_{n}=0,\quad\text{ such that }\quad
  x_{a}\in\mcO_{\mbbK,\mcS}^{\times}
\end{align*}

Such an equation is said to be a \defnfont{proper $\mcS$-unit equation} if one
requires that
\begin{align*}
  x_{i_{0}}+\cdots+x_{i_{r}}\neq0
\end{align*}
for each proper, non-empty subset $\lcb i_{0},i_{1},\ldots, i_{r}\rcb$ of
$\lcb 0,1,\ldots, n\rcb$.

In 1984, Evertse took up a study of $\mcS$-units and the $\mcS$-unit equation
through analyzing the projective height \cite{Ev1}. By bounding the projective
height, he showed that $\mcS$-units obey a remarkable finiteness condition
\textit{loc.~cit.}:
\begin{theorem}
  \label{201302051210}
  If $\mbbK$ is a number field, $\mcS$ a finite set of primes
  of $\mcO_\mbbK$, and $n$ is a fixed positive integer, then
  there are only finitely many projective points
  $X=\[x_{0}:\cdots:x_{n}\]\in\mbbP^{n}\mbbK$ satisfying the proper
  $\mcS$-unit equation:
  \begin{align*}
    x_{0}+\cdots+x_{n}=0.
  \end{align*}
\end{theorem}

With this last ingredient we can now proceed to the
\begin{proof}[Proof of the Rank-Finiteness Theorem]
  For fixed rank $r$, \cite[Prop. 6]{Ban2} (or Proposition \ref{201301161421}) ensures that $\ord\(T\)$ is
  bounded strictly in terms of $r$.  For such $\ord(T)$, let $\mcS$ be the (finite) set of primes in $\mcO_{\ord(T)}$ dividing $\ord(T)$. The Cauchy Theorem
  (\thmref{Cauchy Theorem for Modular Categories}) coupled with \cite[Lem. 1.2]{EG1} then
  implies that $D^{2}$ and $d_{a}$ are $\mcS$-units for
  all simple objects $V_{a}$. Furthermore, the definition of the global dimension
  of the category, $0=D^{2}-d_{0}^{2}-\cdots-d_{r-1}^{2}$, and the condition
  that $d_{a}^{2}$ and $D^{2}$ are real positive algebraic integers for all $a$ implies
  that $\(D^{2},-d_{0}^{2},\ldots, -d_{r-1}^{2}\)$ satisfies a proper
  $\mcS$-unit equation. In
  particular, \thmref{201302051210} shows that there are
  finitely many projective solutions to this equation. Recalling that
  $d_{0}^{2}=1$ allows us to fix the normalization and conclude that there is a
  upper bound on $D^{2}$ and a lower bound on $d_{a}$ for all $a$.

  On the other hand, $\FPdim\(\mcC\)=D^{2}/d_{a}^{2}$ for some simple
  dimension $d_{a}$. Consequently, the lower bound on $d_{a}$ and an upper
  bound on $D^{2}$ imply an upper bound on $\FPdim\(\mcC\)$. The result then follows
  from \corref{201302051239} and the observation that these bounds depend only
  on the rank.
\end{proof}

\subsection{Extensions of rank-finiteness}

A premodular category $\CC$ is called \textbf{modularizable} \cite{Brug1} if there exists a modular category $\DD$ and a dominant ribbon functor $F:\CC\rightarrow \DD$.  For a premodular category $\CC$ the failure of modularity is encoded in the \textbf{M\"uger centralizer} $\CC^\prime$--the symmetric fusion subcategory of $\CC$ generated by all $X$ such that $c_{X,Y}c_{Y,X}=\id_{X\ot Y}$ for all objects $Y$ \cite{M4}.  It is known (\cite[Cor. 3.5]{Brug1},\cite{D1}) $\CC$ is modularizable if and only if $\CC^\prime\cong \Rep(G)$ as a symmetric fusion category, for some finite group $G$ (i.e. $\CC^\prime$ is Tannakian).  Modularization is a special case of a general construction called \textbf{de-equivariantization} \cite{DGNO1} in which one quotients out an action of a finite group $G$ on a category $\CC$ obtaining a new category $\CC_G$.  The inverse operation \emph{equivariantization} gives a way to recover $\CC$ from $\CC_G$.  We can easily extend rank-finiteness to modularizable premodular categories:

\begin{cor}
 There are finitely many modularizable, premodular categories of rank $r$.
\end{cor}
\begin{proof} We follow the notation in \cite{DGNO1}.
 Let $\CC$ be a premodular category such that $\CC^\prime\cong \Rep(G)$ is
Tannakian and $\CC_G=\DD$ its modularization.  Note that the equivariantization
$\DD^G\cong\CC$ by \cite{DGNO1}.  First observe that under the (faithful)
forgetful functor $\CC\rightarrow \DD$ the image of each simple object
$X\in\CC$ is a sum of at most $|G|$ distinct simple objects in $\DD$ (see
\cite[Prop. 2.1]{BuN}). Since the rank of
$\Rep(G)$ is at most $r$, $|G|$ is bounded as a function of $r$ (\cite{L1}).
Therefore, the rank of $\CC_G$ is bounded in terms of $r$.  By Theorem
\ref{Rank Finiteness} there are only finitely many modularizations of rank $r$
premodular categories. On the other hand, each modular category $\DD$ has only
finitely many equivariantizations of bounded rank for groups of bounded order \cite{ENO3}.
\end{proof}

We remark that to extend rank-finiteness to all premodular categories, it is enough to consider only premodular categories with $\CC^\prime\cong sVec$.  Indeed, if $\Rep(G)\subset\CC^\prime$ is a maximal central Tannakian subcategory, then the de-equivariantization $\CC_G$ is premodular with no Tannakian subcategories.  Hence, the resulting category is either modular or has $\CC^\prime\cong sVec$ (cf. \cite{DGNO1}).

We have assumed throughout that $\CC$ has base field $\BC$.  A referee pointed out that rank-finiteness for modular categories holds in more general settings.

Firstly, it is clear that these results are equally valid over an algebraically closed field of characteristic $0$.  Secondly, \cite[Prop. 3.8]{EG2} implies that there are only finitely many rational forms for modular categories, so rank-finiteness holds over any field of characteristic $0$.

Over fields of characteristic $p$, rank-finiteness can be extended for modular categories $\CC$ with $D^2\neq 0$.  By applying the methods of \cite[Section 9]{ENO1}, we see that such a
  category has finitely many lifts to a field of characteristic $0$, and equivalence of the lifts implies equivalence of the original categories.

\section{Asymptotics and Future Directions}\label{asymptotics}

\subsection{Asymptotics}The proof of \thmref{Rank Finiteness} can be naively algorithmized to determine
possible sets of fusion rules for modular categories of a given rank $r$.

Recall that $\mcO_{\mbbK,\mcS}^{\times}$ is a finitely generated abelian
group \cite{We1}.
%
A set of generators for the free part of $\mcO_{\mbbK,\mcS}^{\times}$ is
known as a \defnfont{system of fundamental $\mcS$-units} and there are known
algorithms for computing such a system, e.g. \cite{C2}. We have:
\begin{alg}\ \\
\label{201303131730}
\begin{itemize}
  \item[(0)] Specify the rank, $r$.
  \item[(1)] For each integer $N$ with $1\leq N\leq 2^{2r/3+8}3^{2r/3}$ perform steps 2-6.
  \item[(2)] Form the set of primes $\mcS$,
  consisting of the prime factors of $N$ over $\mbbQ\(\z_{N}\)$.
  \item[(3)] Determine a fundamental system of $\mcS$-units,
  $\e_{1},\ldots, \e_{s-1}$.
  \item[(4)] Solve the exponential Diophantine system:
  \begin{align}
    \label{201302051721}
    1&=\e_{1}^{a_{r,1}}\cdots\e_{s-1}^{a_{r,s-1}}-\Sum_{j=1}^{r-1}\e_{1}^{a_{j,1}}\cdots\e_{s-1}^{a_{j,s-1}},\quad a_{j,k}\in\mbbZ\,.
  \end{align}
  \item[(5)] Set
  $D^{2}=\e_{1}^{a_{s,1}}\cdots\e_{s-1}^{a_{s,s-1}}$ and
  $d_{j}^{2}=\e_{1}^{a_{j,1}}\cdots\e_{s-1}^{a_{j,s-1}}$.
  \item[(6)] Determine the possible sets of fusion rules $N_{a,b}^{c}$ using \lemmaref{201302051302} and the fact that $\FPdim(V_a)\leq\FPdim(\CC)\leq\max\{D^2/d_j^2:0\leq j\leq r-1\}$.
\end{itemize}
\end{alg}

\begin{rmk} Given all possible sets of fusion rules in a given rank we can solve
for all admissible modular data.  The balancing equation (\ref{Balancing})
determines the $S$-matrix given all $d_i$, $\theta_i$ and $N_{ij}^k$, which
Algorithm \ref{201303131730} provides.

We can also effectively decide whether a particular set of fusion rules
corresponds to a modular category, using Tarski's Theorem (see \cite{DHW1}).  We
cannot, however, effectively determine all modular categories in a given rank,
or even count them up to equivalence.
\end{rmk}

 In any case, Algorithm \ref{201303131730} is very inefficient, and does not
admit any obvious improvements for several reasons.

Firstly, we cannot expect a bound on $\ord\(T\)$ that is polynomial in the rank. For
  example, $\ord\(T\)$ for $\mcC=\Rep\(D\mfS_{n}\)$ grows faster than any
  polynomial in the rank of $\mcC$.  Indeed, $\ord(T)=\exp(\mfS_n)=\lcm(1,\ldots,n)\approx e^{n}$, while
  the rank of $\Rep(D\mfS_n)$ is superpolynomial but subexponential, with generating function: $\Pi_{k=1}^\infty (1 - x^k)^{-\sigma(k)}$ where $\sigma(k)=\sum_{d\mid k}d$ \cite{Brit}.  However, for modular categories coming from quantum groups, $\ord(T)$ is linearly bounded in $r$.
 Secondly, the known
algorithms for computing fundamental systems of $\mcS$-units rely on
computing a shortest vector in a lattice, a problem which is known to be
NP-hard.
Thirdly, solving \eqnref{201302051721} is very difficult--our best bound on the number of solutions
is quadruplely exponential:
\begin{prop}
  \label{201302051706}
  For fixed rank $r$ there are at most
  \begin{align*}
    \Sum_{m=1}^{2^{2r/3+8}3^{2r/3}}\(2^{35}r^{2}\)^{r^{3}\(\vph\(m\)^{\log_{2}\(m\)}+\vph\(m\)/2+1\)-r/2}
  \end{align*}
  possible solutions to the dimension equation:
  \begin{align*}
    D^{2}=1+d_{1}^{2}+\cdots+d_{r-1}^{2}\,.
  \end{align*}
\end{prop}
\begin{proof}
  First note that by \cite[Theorem 3]{Ev2}, that there are at most
  $\(2^{35}r^{2}\)^{r^{3}(s+r_{1}+r_{2})-r/2}$ solutions to the proper $\mcS$-unit equation:
  \begin{align*}
    x_{1}+x_{2}+\cdots+x_{r}=1
  \end{align*}
  subject to $x_{j}\leq x_{j+1}$ over a field $\mbbK$, where $s$ is the cardinality of $\mcS$,
  $r_{1}$ is the number of real embeddings of $\mbbK$ and $r_{2}$ is the number
  of conjugate pair complex embeddings.

However, $s$ depends on the prime
  factorization of $\ord\(T\)$. In particular, if $p$ is a rational prime of
  $\ord\(T\)$, and there are at worst $\vph\(\ord\(T\)\)$ primes lying over
  $p$ in $\mbbQ\(\z_{\ord\(T\)}\)$.
  Thus there are at most $\ord\(T\)^{\w\(\ord\(T\)\)}$ primes in
  $\mcS$, where $\w\(m\)$ is defined to be the number of rational prime divisors of
  $m$. Elementary analysis reveals that $\w\(m\)\leq \log_{2}\(m\)$ and so
  $s\leq \vph\(\ord\(T\)\)^{\log_{2}\(\ord\(T\)\)}+r_{1}+r_{2}$ where $r_{1}$
  is the number of distinct real field embeddings of $\mbbK$ into $\mbbC$,
  and $r_{2}$ is the number of conjugate pair complex field embeddings.
  However, it is well-known that a non-trivial cyclotomic field has no real
  embeddings, in particular $r_{2}=\vph\(\ord\(T\)\)/2$ and $s+r_{1}+r_{2}\leq
  \vph\(\ord\(T\)\)^{\log_{2}\ord\(T\)}+\vph\(\ord\(T\)\)/2+1$. Combining
  these two results reveals that an upper bound on the number of possible
  dimension tuples $\(D^{2},1,d_{1}^{2},\ldots,d_{r-1}^{2}\)$ for a rank $r$
  modular category with $T$-matrix of order $\ord\(T\)$ is
  $\(2^{35}r^{2}\)^{r^{3}\(\vph\(\ord\(T\)\)^{\log_{2}\(\ord\(T\)\)}+\vph\(\ord\(T\)\)/2+1\)-r/2}$.
  The result then follows by summing over all possible values of $\ord\(T\)$
  as determined by \propref{201301161421}.
\end{proof}

Low-rank classification suggests that these bounds are far from sharp. We ask:
\begin{question}
  Is there an asymptotic bound on the number of modular categories (up to
  equivalence) in terms of the rank which is better than those implied by
  \propref{201302051706}?
\end{question}

\begin{rmk}\label{EtRem} Etingof \cite{ECo1} has pointed out
that the number of modular categories of
rank $r$ is not polynomially bounded.  His example is as follows:
Consider $V=(\BZ/p)^m$, a vector space over $F_p$ of dimension $m$, where $p>3$
is a prime and $m$ is large.
It is well known that $H^3(V,\BC^*)=S^2V^*\oplus \wedge^3V^*$, (see e.g.
\cite[Lem. 7.6(iii)]{EGO}).
Because of the summand $\wedge^3V^*$, the number of such cohomology classes for
large $m$ is at least $p^{Cm^3}$, for some $C>0$,
even if we mod out by automorphisms (which form a group of order at
most $p^{m^2}$). Now take the category $\Vec(V,\omega)$ of
$V$-graded vector spaces with associativity defined by the cohomology class
$\omega$, and let $Z(V,\omega)$ be the Drinfeld center of such a category.
It is known \cite{ENO3} that such
categories $Z(V,\omega)$ and $Z(V,\omega')$ are braided equivalent if and only
if $\Vec(V,\omega)$ is Morita equivalent to $\Vec(V,\omega')$ via an
indecomposable module category. But the indecomposable module categories over
$\Vec(V,\omega)$ are known to be parameterized
\cite{O5} by subspaces $W\subset V$ and $2$-cochain $\psi$ on $W$ such that
$d\psi=\omega|_W$, up to gauge transformation.
There are at most $p^{m^2}$ subspaces, and freedom in choosing $\psi$ is in
$\wedge^2W^*$, so again there are at most $p^{m^2}$. As $m^3$ dominates $m^2$, we still have at least $p^{Cm^3}$ such categories, even up to Morita equivalence, and hence modular categories up to equivalence. On the other hand, $\FPdim(Z(V,\omega))=p^{2m}$, so the rank is at most
$p^{2m}$. Thus we get that the number of modular categories of rank$\leq r$ is
at least $e^{(c\log(r)^3)}=r^{c\log(r)^2}$, for some $c>0$, which is faster than
any polynomial in $r$.
\end{rmk}

Along similar lines, one might ask
\begin{question}
  Is there an explicit upper bound on $\FPdim\(\mcC\)$ solely in terms of the rank?
\end{question}

\begin{rmk}
  This question seems tractable as the analysis of Evertse shows that the
  projective height of $\[-D^{2}:1:d_{1}^{2}:\cdots:d_{r-1}^{2}\]$ can
  be bounded in terms of field data and hence in terms of $\ord\(T\)$. This
  suggests that the relationship between the FP-dimension and the categorical
  dimension can be combined, as in the proof of \thmref{Rank Finiteness}, to
  study this question.
\end{rmk}

On the other hand, Etingof asked \cite{ECo1}:
\begin{question}
  \label{dimension gap quesiton}
  Can $\abs{D^{2}-1}$ be explicitly bounded away from $0$ in terms of the rank?
\end{question}

\begin{rmk}
  This question can be reduced to the problem of finding a shortest vector by
  exploiting the lattice structure of $\mcO_{\mbbK,\mcS}^{\times}$ under an
  appropriate embedding into Euclidean space.
\end{rmk}

\subsection{Future Directions}

Besides obtaining better asymtotics and classification of low rank modular categories, there are several other open problems.

Physicists propose to use the modular $S,T$ matrices as order parameters
  for the classification of topological phases of matter \cite{LWWY1}. Therefore, a natural
  question is if a modular category is uniquely determined by its modular $S,T$ matrices.
We believe that unitary modular categories are determined by their modular $S,T$ matrices.

The $S,T$ matrices satisfy many constraints as given in Sec. \ref{modulardata}.
It is interesting to characterize realizable modular data in terms of such constraints, in particular,
whether or not admissible modular data is always realizable.

For the application to topological quantum computation, it is important to
  understand the images of the representations of the mapping class groups from a
  modular category.  In particular, when do all representations have finite images? If so, which finite groups arise as such images?
  The property $F$ conjecture says that the representations of all mapping class
  groups from a modular category have finite images if and only if $D^2\in \BZ$ \cite{NR1, RSW}.

\pagestyle{plain}
{\raggedright\printbibliography}

\end{document}